\documentclass[twoside,a4paper,leqno,12pt]{amsproc}
\usepackage[top=30mm,right=30mm,bottom=30mm,left=30mm]{geometry}

\usepackage[pagebackref]{hyperref}
\hypersetup{citecolor=blue, linkcolor=blue, colorlinks=true}
\usepackage{amsmath, amssymb, amsthm, amsfonts, color, comment}
\usepackage{graphicx, float, array, multicol, rotating}
\usepackage{mathtools,mathdots,gensymb}
\usepackage{pgfplotstable} 

\renewcommand{\leq}{\leqslant}
\renewcommand{\geq}{\geqslant}
\renewcommand{\le}{\leqslant}

\newcommand{\Nat}{\mathbb{N}}
\newcommand\Sym{\mathrm{Sym}}

\newenvironment{Myenumerate}{%
\begin{enumerate}[noitemsep, nolistsep, leftmargin=*]
}
{%
\end{enumerate}
}
\usepackage[shortlabels]{enumitem}
\setlist[enumerate]{label=\rm{(\alph*)}}

\theoremstyle{definition}
\newtheorem{definition}{Definition}
\newtheorem{remark}[definition]{Remark}
\theoremstyle{plain}

\newtheorem{question}[definition]{Question}
\newtheorem{theorem}[definition]{Theorem}

\newtheorem{lemma}[definition]{Lemma}

\pgfplotsset{scaled y ticks=false}

\begin{document}

\author{John Bamberg}
\address{J. Bamberg, S.\,P. Glasby, C. E. Praeger: Centre for the Mathematics of Symmetry and Computation, University of Western Australia, 35 Stirling Highway, Perth 6009, Australia.
\emph{E-mail:} {\tt\texttt{\{john.bamberg, stephen.glasby, cheryl.praeger\}@uwa.edu.au}}
}

\author{S.\,P. Glasby}
 
\author{Scott Harper}
\address{S. Harper: School of Mathematics, University of Bristol, Bristol, BS8 1TW, UK.
\emph{E-mail:} {\tt\texttt{scott.harper@bristol.ac.uk}}
}

\author{Cheryl E. Praeger}
\thanks{Acknowledgements: The work on this paper began at the 2018 Research Retreat of the Centre for the Mathematics of Symmetry and Computation. The third author is grateful for the Cecil King Travel Scholarship from the London Mathematical Society and the hospitality of the University of Western Australia; he also thanks the Engineering and Physical Sciences Research Council and the Heilbronn Institute for Mathematical Research for their financial support. The problem forms part of an Australian Research Council Discovery Project.\newline
2010 Math Subject Classification: 20B30, 05A15, 68W20.
}
 
\title[Permutations with orders coprime to a given integer]{Permutations with orders\\ coprime to a given integer}

\subjclass[2010]{}
\date{\today}

\begin{abstract}
Let $m$ be a positive integer and let $\rho(m,n)$ be the proportion of permutations of the symmetric group $\Sym(n)$ whose order is coprime to $m$. In 2002, Pouyanne proved that
$\rho(n,m)n^{1-\frac{\phi(m)}{m}}\sim \kappa_m$ where $\kappa_m$ is a complicated
(unbounded) function of $m$.
We show that there exists a positive constant $C(m)$ such that,
for all $n \geq m$,
\[
C(m) \left(\frac{n}{m}\right)^{\frac{\phi(m)}{m}-1} \leq \rho(n,m) \leq \left(\frac{n}{m}\right)^{\frac{\phi(m)}{m}-1}
\]
where $\phi$ is Euler's totient function. 
\end{abstract}

\maketitle

\section{Introduction}\label{s:intro}
In a series of papers between 1965 and 1972, Erd\H{o}s and Tur\'an initiated a systematic study of probabilistic aspects of group theory (see, for example, \cite{ref:ErdosTuran65}). One topic which has been of particular interest since this time is the distribution of element orders in finite symmetric groups, and their most relevant work for us on this topic began in  \cite{ref:ErdosTuran67, ref:ErdosTuran67b} where they studied the proportion $p_{\neg m}(n)$ of elements in $\Sym(n)$ with no cycle of length divisible by a fixed prime $m$. Erd\H{o}s and Tur\'an  obtained an explicit formula for $p_{\neg m}(n)$ and determined the limiting proportion, as $n$ grows, as
\begin{equation}\label{eq:prime}
 p_{\neg m}(n) = k(m) \left(\frac{n}{m}\right)^{-\frac{1}{m}} + O(n^{-1-\frac{1}{m}}),
\end{equation}
where $k(m) = \Gamma\left(1-\frac{1}{m}\right)^{-1}$, noting that $\pi^{-1/2} \leq k(m) <1$ \cite[Sections~3 and~4]{ref:ErdosTuran67}. Although $m$ was assumed to be a prime in \cite{ref:ErdosTuran67}, the formula for $p_{\neg m}(n)$ in \eqref{eq:prime} holds for an arbitrary positive integer $m$, see \cite{ref:glasby}, and their asymptotic arguments can be extended to give explicit  convergence bounds \cite[Theorem 2.3(b)]{ref:BLGNPS02}, again for arbitrary $m$. These explicit bounds, together with analogous results for alternating groups \cite[Section 3]{ref:BLGNPS02}, were used to analyse algorithms for constructing transpositions and 3-cycles \cite[Section 6]{ref:BLGNPS02}, procedures used as components of the constructive recognition algorithms for black-box alternating and symmetric groups in \cite{ref:BLGNPS03}. Many other authors have also considered the proportion $p_{\neg m}(n)$, see for example \cite{ref:BG89, ref:BMW00, ref:NPPY11} and the discussion in \cite{ref:NPS13}.  

Let us introduce the specific topic of interest for this paper. For positive integers $n$ and $m$, let $R(n,m)$ be the set of elements of $\Sym(n)$ whose order is coprime to $m$, and write
\[
 \rho(n,m) \coloneqq \frac{|R(n,m)|}{n!}.
\]
The proportion $\rho(n,m)$ is equal to the proportion  $p_{\neg m}(n)$ of Erd\H{o}s and Tur\'an discussed above if and only if $m$ is a prime power.  Moreover, in \cite[Lemma~II]{ref:ErdosTuran67}, Erd\H{o}s and Tur\'an demonstrate that if $n$ is sufficiently large and $m$ is the product of two distinct primes $p$ and $q$ satisfying $(\log{n})^{3/4} \leq p,q \leq 10\log{n}/\log\log{n}$, then
\begin{equation}\label{eq:two_primes}
\rho(n,m) = n^{-\frac{1}{p}-\frac{1}{q}}(1+O(\log^{-\frac{1}{2}}{n})).
\end{equation}

Pouyanne~\cite[Proposition, p.\,7]{ref:Pouyanne02} used a singularity analysis
on the generating function $C(x)=\sum_{i\geq  0}\rho(n,m)X^m$ for $\rho(n,m)$
to give an asymptotic value of $\rho(n,m)$ for arbitrary $m$. He gives a
nice proof that $\rho(n,m)n^{1-\phi(m)/m}\sim\kappa_m$ where $\kappa_m$ is a
function of~$m$ involving Gamma and M\"obius functions,
see~\eqref{E:kappam}. Unfortunately the elusive
nature~\cite[Figure~1]{ref:Pouyanne02} of
$\kappa_m$ makes it hard to apply this result.
In particular, upper and lower bounds $\rho(n,m)$ cannot be extracted from
the asymptotics in~\cite{ref:Pouyanne02}, and our major contribution is to
bound the quantity $\lambda_m:=\kappa_m/m^{1-\phi(m)/m}$, where $\phi$ is
Euler's totient function.
We need these bounds for applications to randomised
(1-sided Monte Carlo) permutation group algorithms where explicit bounds
on the probability/proportions are required to assign explicit upper
bounds on the probability that the algorithm returns an incorrect answer,
i.e. to prove that it is a Monte Carlo algorithm. Examples of the
use of such probability bounds for exhibiting a Monte Carlo algorithm,
and analysing its complexity, are given for example in~\cite{ref:BLGNPS03}.
Specifically, our algorithm for testing whether a subgroup $\langle X\rangle$ of $\Sym(n)$ contains the alternating group $\textup{Alt}(n)$ either returns
the answer ``Yes'' with no chance of error, or returns an answer ``No'' with
a (preset arbitrarily) small probability of error, say $10^{-6}$.

The set $\pi(m)$ of prime divisors of $m$ is significant as $\rho(n,m)=\rho(n,m')$ and $\phi(m)/m=\phi(m')/m'$ when $\pi(m)=\pi(m')$. Given this fact, we will henceforth assume that $m$ is \emph{square-free}. We implicitly also assume that the primes in $\pi(m)$ are
at most $n$, since $\rho(n,m)=\rho(n,mp)$ for primes $p>n$. With this in mind,
and observing that $\frac{\phi(m)}{m}-1\leq 0$, we now present our main result. 

\begin{theorem}\label{t:theorem}
Let $m$ be a positive square-free integer. There exists a positive constant $C(m)$ such that, for all $n \geq m$,
\[
C(m) \left(\frac{n}{m}\right)^{\frac{\phi(m)}{m}-1} \leq \rho(n,m) \leq \left(\frac{n}{m}\right)^{\frac{\phi(m)}{m}-1}.
\]
\end{theorem}

The exponent $\frac{\phi(m)}{m}-1$ in Theorem~\ref{t:theorem} is negative, and hence $\lceil\frac{n}{m}\rceil^{\frac{\phi(m)}{m}-1} \leq (\frac{n}{m})^{\frac{\phi(m)}{m}-1}$ and $\lfloor\frac{n}{m}\rfloor^{\frac{\phi(m)}{m}-1} \geq (\frac{n}{m})^{\frac{\phi(m)}{m}-1}$, for $n\geq m$.  Thus, in order to prove Theorem~\ref{t:theorem}  it is sufficient to prove that 
\begin{equation}\label{eq:E}
C(m) \left\lfloor\frac{n}{m}\right\rfloor^{\frac{\phi(m)}{m}-1} \leq \rho(n,m) \leq \left\lceil\frac{n}{m}\right\rceil^{\frac{\phi(m)}{m}-1}.
\end{equation}
We prove these inequalities in Section~\ref{s:proof}. In fact the upper bound holds for $n\geq1$. We conclude with a conjecture in Section~\ref{sec:conj} based on computational evidence.

First we make a few remarks concerning the constant $C(m)$ and links between Theorem~\ref{t:theorem} and the results  \eqref{eq:prime} and \eqref{eq:two_primes}.
\begin{remark}\leavevmode
\begin{Myenumerate}
\item We prove Theorem~\ref{t:theorem} with the constant
\begin{equation}
 C(m) \coloneqq \min\{ \rho(n, m) \mid m \leq n \leq 2m-1 \}.  \label{eq:C}
\end{equation}
In particular, if $m$ is a prime then $C(m) = 1-\frac{1}{m}$.

\item If an element of $\Sym(n)$ has order coprime to $m$, then the length of each of its cycles is certainly not divisible by $m$. Hence, we have the upper bound $\rho(n,m)\leq p_{\neg m}(n)=\prod_{i=1}^{\lfloor\frac{n}{m}\rfloor}(1-\frac{1}{im})$ by~\cite{ref:glasby}. However, this bound grows too quickly as remarked on in (c).

\item If $m$ is prime, then the exponent is $\frac{\phi(m)}{m}-1 = \frac{m-1}{m}-1 = -\frac{1}{m}$, and we obtain from  Theorem~\ref{t:theorem} the result \eqref{eq:prime}, apart from determining the constant $k(m)$. In fact, the exponent  $\frac{\phi(m)}{m}-1$ is equal to  $-\frac{1}{m}$ if and only if $m$ is a power of a prime, and in all other cases the exponent is strictly less than   $-\frac{1}{m}$. In other words, if $m$ is divisible by at least two primes then $\rho(n,m)$ grows more slowly, as $n$ increases, than $p_{\neg m}(n)$ does. 

\item Suppose $m=pq$ where $p<q$ are primes. Then $\frac{\phi(m)}{m}-1 = -\frac{1}{p}-\frac{1}{q} + \frac{1}{pq}$ and Theorem~\ref{t:theorem} appears to differ from \eqref{eq:two_primes} by a multiplicative factor of $n^{1/pq}$. However, in our context $m$ is fixed and $n$ increases without bound, whereas Erd\H{o}s and Tur\'an assume for \eqref{eq:two_primes} that both $p$ and $q$ are bounded:
\begin{equation}\label{eq:ET}
  (\log n)^{3/4} \leq p<q \leq \frac{10\log n}{\log\log n}.
\end{equation}
Thus, both $m$ and $n$ are assumed to increase in~\eqref{eq:two_primes}. The apparent inconsistency can be resolved by showing that~\eqref{eq:ET} implies
\[
n^{\tfrac{1}{pq}} = 1 + O((\log n)^{-1/2}).
\]
For an upper bound, from~\eqref{eq:ET} we have
\[
n^{1/(pq)}\leq n^{(\log n)^{-3/2}} = n^{(\log n)^{-1} (\log n)^{-1/2}}
= e^{(\log n)^{-1/2}} = 1 + O((\log n)^{-1/2}).
\]
For a lower bound we show 
\[
  n^{\tfrac{1}{pq}}\geq n^{(\log \log n)^2/(100(\log n)^2)} \geq 1 + O((\log n)^{-1/2}).
\]
Establishing the last inequality is the same as bounding (above) the function
\[
  f(n) \coloneqq (n^{x(\log n)^{-1}} - 1)(\log n)^{1/2}\quad\textup{where}\quad x=\frac{(\log\log n)^2}{100\log n}.
\]
Rewriting $f(n)$ using the identity $n^{(\log n)^{-1}}=e$ gives
\[
  f(n) = (e^x-1)(\log n)^{1/2}.
\]
Since $x\to 0$ as $n\to\infty$, we can choose $n$ large enough so that $x<1/2$. However,
$0\leqslant e^x-1< 2x$ for $0\leqslant x<1/2$ so
\[
  0\leq f(n) < 2x (\log n)^{1/2} = \frac{(\log \log n)^2}{50 (\log n)^{1/2}}.
\]
Hence $f(n)\to 0$ as $n\to\infty$, so $f(n)$ is bounded above as claimed.

\item The proofs by  Erd\H{o}s and Tur\'an of results such as \eqref{eq:prime} and \eqref{eq:two_primes} draw heavily on tools from complex analysis. In \cite[Section~5]{ref:ErdosTuran67}, Erd\H{o}s and Tur\'an state that it would be desirable to obtain a proof of \eqref{eq:two_primes} using more direct means:
\begin{quote}
     \textit{``A more direct (real-variable or algebraic) approach to the determination of this coefficient would be desirable.''}
\end{quote}
The proof of Theorem~\ref{t:theorem} is principally algebraic: we determine and exploit a recursive formula for $\rho$.

\item In a different direction, restricting $m$ to a prime number and determining the proportion $\rho(G,m)$ of elements of an arbitrary finite group $G$ whose order is coprime to $m$ has been the subject of papers by many authors. For example, see \cite{ref:IsaacsKantorSpaltenstein95} when $G$ is a permutation group of degree $n$ and see \cite{ref:BabaiGuestPraegerWilson13, ref:GuestPraeger12, ref:GuralnickLubeck01} when $G$ is a finite simple classical group.

\item The set $\Sym(n)^{(m)}=\{\pi^m\mid \pi\in\Sym(n)\}$ of $m$th powers, and
  its cardinality, have been extensively studied, e.g.~\cite{ref:MP,ref:P}. As
  every permutation of order coprime to~$m$ is an $m$th power, we have
  $R(n,m)\subseteq\Sym(n)^{(m)}$. The containment is proper in general,
  for example
  $(1,3)(2,4)\in\Sym(4)^{(2)}\setminus R(4,2)$. However, if $m$~divides
  the exponent~$e$ of $\Sym(n)$ and $\gcd(m,e/m)=1$, then $R(n,m)=\Sym(n)^{(m)}$.
  Hence, one may guess that $|\Sym(n)^{(m)}|$ and $|R(m,n)|$ have
  the same asymptotic density. This follows from~\cite{ref:MP,ref:P}
  and~\cite{ref:Pouyanne02}.
\end{Myenumerate}
\end{remark}

\section{Proof of Theorem~\ref{t:theorem}} \label{s:proof}
For the remainder of the paper, fix $m$ as a square-free positive integer. Recall that $R(n,m)$ is the set of elements in $\Sym(n)$ of order coprime to $m$. Since $m$ is fixed we will write  $R(n):= R(n,m)$ and similarly (except in some formal statements) we write $\rho(n):=\rho(n,m)$. Additionally, we denote the greatest common divisor of integers $c$ and $d$ by $(c,d)$, and we write   
\[
  \Phi = \Phi(m) \coloneqq \{ 1 \leq i \leq m \mid (i,m)=1 \},
\] 
noting that $\phi \coloneqq \phi(m) = |\Phi|$.

The following lemma generalises \cite[Lemma~2.1]{ref:BLGNPS02}.  For convenience, we adopt the convention that $\rho(0) = 1$.

\begin{lemma}\label{l:recursion}
The following recursive formula holds for integers $n \geq m>0$, 
\[ 
  n\rho(n) = (n-m)\rho(n-m) + \sum_{k \in \Phi}^{} \rho(n-k). 
\]
\end{lemma}

\begin{proof}
The permutations $x \in R(n)$ can be enumerated according to the length $k$ of the cycle containing the point $1$. The number of choices for the cycle $(1,i_2,\dots,\kern-1pt i_k)$ of $x$ is $(n-1)(n-2)\cdots(n-k+1)$.  Note that $(k,m)=1$ and that the permutation induced by $x$ on the $n-k$ points outside $\{1,i_2,\dots,i_k\}$ lies in $R(n-k)$. Thus
\[
  |R(n)| = \sum_{\substack{1 \leq k \leq n \\ (k,m)=1}}(n-1)(n-2)\cdots(n-k+1) |R(n-k)|.
\]
Dividing this equation by $(n-1)!$, and noting that $|R(a)|=a!\rho(a)$ for all $a \in \Nat$, we obtain
\[
  n\rho(n) = \sum_{\substack{1\leq k \leq n \\ (k,m)=1}} \rho(n-k).
\]
Replacing $n$ above with $n-m$ and observing that
$(k+m,m)=(k,m)$ yields
\[
  (n-m)\rho(n-m) = \sum_{\substack{1\leq k \leq n-m \\ (k,m)=1}} \rho(n-m-k) = 
  \sum_{\substack{m+1\leq k \leq n \\ (k,m)=1}} \rho(n-k).
\]
Subtracting these two equations gives
\[
  n\rho(n) - (n-m)\rho(n-m) = \sum_{k \in \Phi} \rho(n-k). \qedhere
\]
\end{proof}

We now present a technical lemma which will be of use in the proof of Theorem~\ref{t:theorem}. 

\begin{lemma}\label{l:taylor}
Let $y$ and $a$ be real numbers such that $-1 < y < 0$ and $a \geq 2$.
Then 
\[
0 < 1 - \frac{y+1}{a}\left(1-\frac{y}{a}\right) \leq \left(\frac{a-1}{a}\right)^{y+1} 
< 1- \frac{y+1}{a}.
\]
\end{lemma}

\begin{proof}
Let $x=-1/a$ and $x_0=-1/2$, and note that $x_0 \leq x < 0$. We seek upper and lower bounds for $f(x) := (1+x)^{y+1} = (\frac{a-1}{a})^{y+1}$. As $|x|<1$, the binomial series below converges absolutely
\[
f(x) = \sum_{i \geq 0} \binom{y+1}{i} x^i.
\]
Since $-1<y < 0$, for each $i>0$, the binomial coefficient 
\[
\binom{y+1}{i} = \frac{(y+1)y(y-1)\cdots(y-(i-2))}{i!}
\] 
has $i-1$ negative factors. Hence, the product $\binom{y+1}{i}x^i$ is 
negative for each $i>0$. Therefore, 
\[
f(x) = \sum_{i \geq 0} \binom{y+1}{i} x^i < 1 + (y+1)x
= 1 - \frac{y+1}{a}
\]
yielding the desired upper bound.

Now we consider the lower bound. Temporarily we assume that $i \geq 2$. Since $(y-1)\cdots(y-(i-2))$ has $i-2$ negative factors, the product $(y-1)\cdots(y-(i-2))x^{i-2}$ is positive for each $i\geq 2$. Hence,
\[
 0 < \prod_{j=1}^{i-2} (y-j)x = \prod_{j=1}^{i-2} (j-y)(-x) \leq \prod_{j=1}^{i-2} (j+1)(-x_0) =(i-1)!(-x_0)^{i-2}.
\]
This in turn shows that
\[
 0 > \binom{y+1}{i}x^i = \frac{(y+1)y(y-1)\cdots(y-(i-2))x^i}{i!} \geq \frac{(y+1)y(-x_0)^{i-2}x^2}{i}.
\]
Taking the terms with $0 \leq i < 2$, together with the above lower bound for the sum of the terms with $i \geq 2$, gives
\begin{align*}
 f(x)&\geq1+(y+1)x+\sum_{i\geq2}\frac{(y+1)y(-x_0)^{i-2}x^2}{i}\\ &=1+(y+1)x+\frac{(y+1)y}{x_0^2}\left(\sum_{i\geq2}\frac{(-x_0)^{i}}{i}\right)x^2.
\end{align*}
Now $\sum_{i \geq 1} \frac{(-x_0)^{i}}{i} = -\log(1+x_0)$, and hence, 
since $x_0 = -1/2$, we have 
\[ 
 x_0^{-2}\sum_{i\geq2}\frac{(-x_0)^{i}}{i}=x_0^{-2}(x_0-\log(1+x_0)),
\]
and this lies in the open interval $(0,1)$. Then since $(y+1)yx^2 < 0$, we obtain the desired lower bound
\[
 f(x)=(1+x)^{y+1}>1+(y+1)x+(y+1)yx^2 = 1 - \frac{y+1}{a}\left(1-\frac{y}{a}\right).
\]
Finally, since $-1<y<0$ and $a\geq 2$, this lower bound is positive.
\end{proof}

We now prove our main result.
\begin{proof}[Proof of Theorem~\ref{t:theorem}]
The result is true when $m=1$ and $C(1)=1$.
Suppose $n\geq m\geq2$. Recall the notation $\Phi = \Phi(m)$ and $\phi = |\Phi|$, and write
\[
 y \coloneqq \frac{\phi}{m}-1.
\]
Observe that $-1 < y < 0$. In addition, for $0 \leq i \leq m-1$, write
\begin{equation}
  x_i = |\{ k \in \Phi \mid k < m-i \}| \label{eq:x} \quad\textup{and}\quad
  y_i = |\{ k \in \Phi \mid k \leq i \}|.
\end{equation}
Note that $x_i\leq m-i-1$, $y_i\leq i$, $x_i + i \geq \phi(m)$ and $y_i +(m-i)\geq \phi(m)$. In summary
\begin{equation}\label{eq:bounds}
\phi(m)-i\leq x_i\leq m-i-1\quad\textup{and}\quad
\phi(m)-m+i\leq y_i\leq i.
\end{equation}

We begin by proving the required upper bound, namely
\begin{equation}
  \rho(n) \leq \left\lceil \frac{n}{m} \right\rceil^y\qquad
  \textup{for $n\geq m\geq2$}. \label{eq:upper}
\end{equation}
Although we do not require it for this proof, the upper bound
above holds trivially if $1 \leq n \leq m$ as then $\rho(n) \leq 1 = \left\lceil\frac{n}{m}\right\rceil^y=1$.
We proceed by induction on $n$. 
Now let $n \geq m+1$, so that $a \coloneqq \left\lceil\frac{n}{m}\right\rceil \geq 2$. 
Write $n=am-b$, and note that $0 \leq b \leq m-1$. Assume the upper bound in \eqref{eq:upper} holds for all positive integers strictly less than $n$. By Lemma~\ref{l:recursion},   
\[
  \rho(am-b) = \frac{(a-1)m-b}{am-b} \rho((a-1)m-b) + \frac{1}{am-b} \sum_{k \in \Phi}^{} \rho(am-b-k).
\]
By the inductive hypothesis, $\rho((a-1)m-b) \leq (a-1)^y$. Similarly, for each $k \in \Phi$, if $k < m-b$ then $am-b-k > (a-1)m$ so by induction $\rho(am-b-k) \leq a^y$, and if $k \geq m-b$ then $\rho(am-b-k) \leq (a-1)^y$. Therefore, using the definition of $x_i$ in \eqref{eq:x}, we obtain
\begin{align*}  
  \rho(am-b) &\leq \frac{(a-1)m-b}{am-b}(a-1)^y + \frac{x_ba^y + (\phi-x_b)(a-1)^y}{am-b} \\
             &=     a^y \left( \left( \frac{a-1}{a} - \frac{b/a}{am-b} \right)\left(\frac{a-1}{a}\right)^y + \frac{x_b + (\phi-x_b)\left(\frac{a-1}{a}\right)^y}{am-b}\right) \\
             &=     a^y \left( \left( \frac{a-1}{a} \right)^{y+1} \left( 1 - \frac{b - a\phi + ax_b}{(a-1)(am-b)}\right) + \frac{x_b}{am-b}\right).
\end{align*}
By Lemma~\ref{l:taylor}, $( \frac{a-1}{a})^{y+1} < 1 - \frac{y+1}{a}$, and as $y+1=\frac{\phi}{m}$ and $a \geq 2$, we have
\[
  \rho(am-b) \leq  a^y Y\textup{ where } Y= \left(1-\frac{\phi}{am}\right) \left( 1 - \frac{b - a\phi + ax_b}{(a-1)(am-b)}\right) + \frac{x_b}{am-b}.
\]
We want to show that $Y\leq 1$, so we write $Y=1-Y_0$ where $Y_0$ is an algebraic fraction in $a,b,x_b,m,\phi$. It suffices, therefore, to show that $Y_0\geq0$ for all input values satisfying $a\geq2$, $0\leq b<m$, and $\phi\leq\min\{b+x_b,m\}$ \emph{c.f.}~\eqref{eq:bounds}. We use a computer to factor $Y_0$ giving
 \[
 Y_0= 1 - Y   =  \frac{(m-\phi)(b+x_b-\phi)}{m(a-1)(am-b)} \geq 0.
\]
Thus $Y\leq 1$ and hence $\rho(am-b) \leq a^y$, proving the upper bound~\eqref{eq:upper} for all $n \geq 1$.

We now turn to the lower bound. Recall the definition of $C := C(m)$ in \eqref{eq:C}, and  note that $C > 0$ since $\rho(n) > 0$ for all $n \geq 1$. We will prove that, 
\begin{equation}
  \rho(n) \geq C \left\lfloor \frac{n}{m} \right\rfloor^y
  \qquad\textup{for $n\geq m\geq2$.}\label{eq:lower}
\end{equation}
As for the proof of the upper bound, we use induction on $n$.  Observe that if $m \leq n \leq 2m-1$, then $\left\lfloor\frac{n}{m}\right\rfloor = 1$, and hence  
$\rho(n) \geq C = C \left\lfloor\frac{n}{m}\right\rfloor^y$
holds by \eqref{eq:C}. Now suppose $n \geq 2m$. Then $a \coloneqq \left\lfloor\frac{n}{m}\right\rfloor \geq 2$. Write $n=am+b$, and note that $0 \leq b \leq m-1$. (Be aware that the definitions of $a$ and $b$ differ from their definitions in the proof of the upper bound.) Assume that the lower bound~\eqref{eq:lower} holds for all positive integers strictly less than~$n$. By Lemma~\ref{l:recursion},
\[
  \rho(am+b) = \frac{(a-1)m+b}{am+b} \rho((a-1)m+b) + \frac{1}{am+b} \sum_{k \in \Phi}^{} \rho(am+b-k).
\]
By the inductive hypothesis, $\rho((a-1)m+b) \geq C(a-1)^y$. Similarly, for each $k \in \Phi$, if $k \leq b$ then $am+b-k\geq am$ so by induction, $\rho(am+b-k) \geq Ca^y$, and if $k > b$ then $am > am+b-k\geq  (a-1)m$ so by induction $\rho(am+b-k) \geq C(a-1)^y$.  Therefore, using the definition of $y_b$ in \eqref{eq:x}, we obtain
\begin{align*}
  \rho(am+b) &\geq  C \left( \frac{(a-1)m+b}{am+b} (a-1)^y + \frac{y_ba^y + (\phi-y_b)(a-1)^y}{am+b} \right) \\
             &=     Ca^y \left( \left( \frac{a-1}{a} + \frac{b/a}{am+b} \right)\left(\frac{a-1}{a}\right)^y + \frac{y_b + (\phi-y_b)\left(\frac{a-1}{a}\right)^y}{am+b}\right) \\
             &=     Ca^y \left( \left( \frac{a-1}{a} \right)^{y+1} \left( 1 + \frac{b + a\phi - ay_b}{(a-1)(am+b)}\right) + \frac{y_b}{am+b}\right).
\end{align*}
By Lemma~\ref{l:taylor}, since $a \geq 2$, $y=\frac{\phi}{m}-1$ and $-1 < y < 0$, we have $\left(\frac{a-1}{a}\right)^{y+1} > 1-\frac{y+1}{a}\left( 1-\frac{y}{a}\right)$, so
\[
  \rho(am+b) \geq  Ca^y \left( \left(1-\frac{\phi}{am} \left( 1 + \frac{m-\phi}{am} \right) \right)  \left( 1 + \frac{b + a\phi - ay_b}{(a-1)(am+b)}\right) + \frac{y_b}{am+b}\right).
\]
Write the above expression as $Ca^yY$ where $Y$ is an algebraic fraction in $a,b,y_b,m,\phi$. We want to show that $Y\geq 1$, so we write $Y=1+Y_0$. It suffices, therefore, to show that $Y_0\geq0$ for all input values satisfying $a\geq2$, $0\leq b<m$, $m\geq\phi$ and $\phi -m+b\leq y_b\leq b$ (see \eqref{eq:bounds}). We use a computer to factor $Y_0$ giving
\[
  Y_0=Y-1=
  \frac{(m-\phi)(am(b-y_b)+\phi(y_b-b+m-\phi))}{m^2 a (a-1)(am+b)}
  \geq0.
\]
Therefore, $\rho(am+b) \geq Ca^yY\geq Ca^y$ and the claim in \eqref{eq:lower} holds for all $n \geq m$. This establishes the lower bound and completes the proof of the theorem.
\end{proof}

\section{Computational evidence}\label{sec:conj}
Let $n \geq m > 1$ and assume that $m$ is square-free. First suppose that $m$ is prime. Recall that $p_{\neg m}(n)$ is the proportion of elements in $\Sym(n)$ with no cycle of length divisible by $m$, so $p_{\neg m}(n)=p_{\neg m}(n+i)$ for $0\leq i<m$. Since $\rho(n,m)=p_{\neg m}(n)$, it follows that for all $a \geq 1$,
\begin{equation}\label{eq:equal_prime}
\rho(am,m) = \rho(am+1,m) = \cdots = \rho(am+(m-1),m).
\end{equation}
Moreover, in this case (since $m$ is prime),
\begin{equation}\label{eq:constant_prime_rho}
\rho(n,m) = k(m) \left(\frac{n}{m}\right)^{\frac{\phi(m)}{m}-1} + O(n^{\frac{\phi(m)}{m}-2}),
\end{equation}
where $k(m)=\Gamma(1-\frac{1}{m})^{-1}$, noting that $\pi^{-1/2} \leq k(m) < 1$ (see \cite[Sections~3 and~4]{ref:ErdosTuran67} and \cite[Theorem~2.3]{ref:BLGNPS02}).

In this final section we investigate the extent to which an analogue of the relationship in \eqref{eq:constant_prime_rho} holds for general positive integers $m$. We do this by presenting some computational evidence which led the authors to the statement of Theorem~\ref{t:theorem} and to Question~\ref{con:conjecture} below.

The recursive formula for $\rho$ in Lemma~\ref{l:recursion} provides an efficient means of computing $\rho(n,m)$ from the values $\rho(0,m)$, $\rho(1,m)$, \dots, $\rho(m-1,m)$. In Figures~\ref{f:plot6}--\ref{f:plot30} we fix the value of $m$ as $6$, $15$ and $30$, respectively, and we plot 
\[
  f(n,m) \coloneqq \rho(n,m) \cdot \left(\frac{n}{m}\right)^{1-\frac{\phi(m)}{m}}
\] 
against $n$ for many values of $n$ greater than $m$.

\begin{figure}[H]
\begin{tikzpicture}
\begin{axis}[width=0.52\textwidth,xlabel=$n$,ylabel=$\rho(n)(n/6)^{1-\phi(6)/6}$, 
    xmin = 0, xmax=2000, xtick distance=500, ymin=0.05, ytick distance = 0.05,
	yticklabel style={
            /pgf/number format/fixed,
            /pgf/number format/precision=2,
            /pgf/number format/fixed zerofill
        }]
  \pgfplotstableread[col sep=comma]{d6_2000}\datasix;
  \addplot [color=blue, 
           only marks, 
           mark=*, 
           mark size=0.5pt] table [x=X, y=Y]{\datasix};
\end{axis}
\end{tikzpicture}
\caption{Plot of $f(n,6)$ versus $n$ for $7 \leq n \leq 2000$.} \label{f:plot6}
\end{figure}
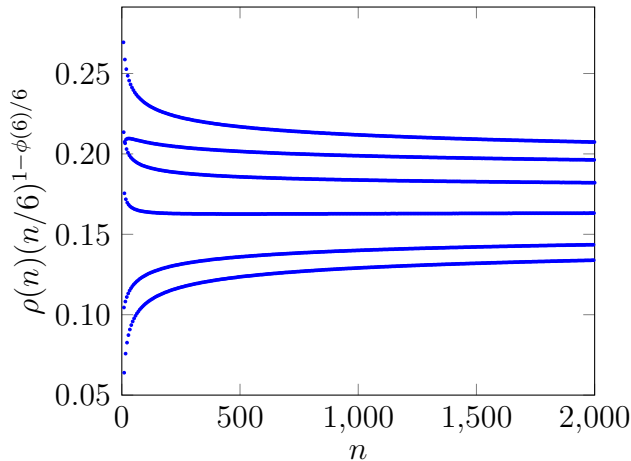

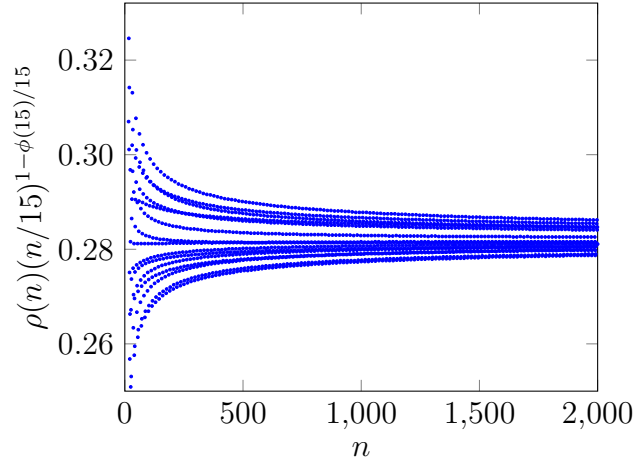
\begin{figure}[H]
\begin{tikzpicture}
\begin{axis}[width=0.52\textwidth,xlabel=$n$,ylabel=$\rho(n)(n/15)^{1-\phi(15)/15}$, xmin = 0, xmax=2000, xtick distance=500, ymin=0.25, ytick distance = 0.02,
	yticklabel style={
            /pgf/number format/fixed,
            /pgf/number format/precision=2,
            /pgf/number format/fixed zerofill
        }]
  \pgfplotstableread[col sep=comma]{d15}\datafifteen;
  \addplot [color=blue, 
           only marks, 
           mark=*, 
           mark size=0.5pt] table [x=X, y=Y]{\datafifteen};
\end{axis}
\end{tikzpicture}
\caption{Plot of $f(n,15)$ versus $n$ for $16 \leq n \leq 2000$.} \label{f:plot15}
\end{figure}

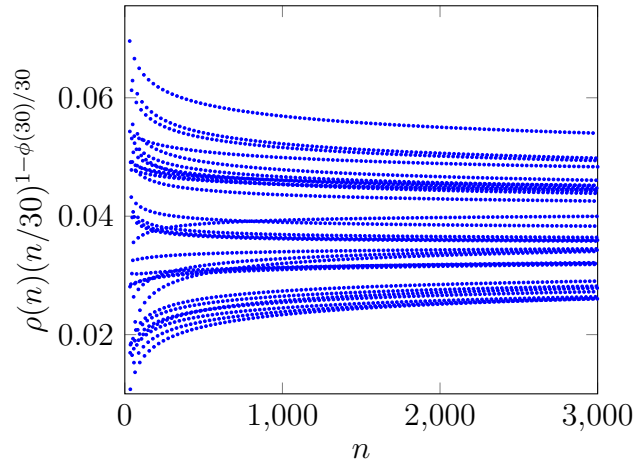
\begin{figure}[H]
\begin{tikzpicture}
\begin{axis}[width=0.52\textwidth,xlabel=$n$,ylabel=$\rho(n)(n/30)^{1-\phi(30)/30}$, xmin = 0, xmax=3000, xtick distance=1000, ymin=0.01, ytick distance = 0.02,
	yticklabel style={
            /pgf/number format/fixed,
            /pgf/number format/precision=2,
            /pgf/number format/fixed zerofill
        }]
  \pgfplotstableread[col sep=comma]{d30}\datafifteen;
  \addplot [color=blue, 
           only marks, 
           mark=*, 
           mark size=0.5pt] table [x=X, y=Y]{\datafifteen};
\end{axis}
\end{tikzpicture}
\caption{Plot of $f(n,30)$ versus $n$ for $31 \leq n \leq 3000$.} \label{f:plot30}
\end{figure}

It is evident from Figures~\ref{f:plot6}--\ref{f:plot30} that~\eqref{eq:equal_prime} does not hold if $m$ is composite.  Figures~\ref{f:plot6}--\ref{f:plot30} suggest that  for fixed $0 \leq b < m$ the function $f(n,m)$ is either increasing or decreasing as $n\to\infty$ with $n\equiv b\pmod m$, and moreover that the limit is independent of $b$.
This would imply~\cite[Proposition, p.\,7]{ref:Pouyanne02} and give even sharper
bounds than in our main theorem as we explain below.
Pouyanne~\cite[Proposition, p.\,7]{ref:Pouyanne02} defined
a constant $\kappa_m$ (for not necessarily square-free $m$) as follows:
\begin{equation}\label{E:kappam}
  \kappa_m=\frac{1}{\Gamma\left(\frac{\phi(m)}{m}\right)}\prod_{d\mid m}d^{-\frac{\mu(d)}{d}}\quad\textup{where}\quad
  \mu(d)=\begin{cases}(-1)^{|\pi(d)|}&\textup{if $d$ is square-free,}\\
  0&\textup{otherwise.}
  \end{cases}
\end{equation}
Thus $f(n,m)\sim\lambda_m:=\kappa_m/m^{1-\phi(m)/m}$ as $n\to\infty$
paraphrases Pouyanne's result.
Theorem~\ref{t:theorem} proves that $C(m)\leq\lambda_m\leq1$.
Figures~1--3 show that the convergence as $n\to\infty$
of $f(n,m)$ to $\lambda_m$ can be very slow. Computational evidence
suggests that the sequence $\left( f(am+b,m)\right)_{a=0}^\infty$ is eventually
monotonic. This leads us to the following question.

\begin{question}\label{con:conjecture}
  Let $m$ be a positive square-free integer. Does there exists an
  integer~$a_0$ such that for each $b$ the sequence $\left( f(am+b,m)\right)_{a\geq a_0}$
  is monotonic? 
\end{question}

\begin{remark}
  If this is true, then for $a\geq a_0$, $f(am+b,m)$ is bounded
  between $f(a_0m+b,m)$ and $\lambda_m=\kappa_m/m^{1-\phi(m)/m}$. When
  $m=p$ is prime and $0\leq b\le\frac{p-1}2$, Theorem~\ref{T:monotone} below
  shows $\lambda_p\leq f(ap+b,p)\leq 1-\frac1p$ and for all $a\geq1$. This improves~\eqref{eq:prime}.
\end{remark}

\begin{remark}
We used the optimised \textsc{Magma}~\cite{ref:BCP97}
code in~\cite{ref:ScottStephen}, and the recurrence in
Lemma~\ref{l:recursion}, to compute values of $\rho(n,m)$ for $n$ up to $10^5$
and $m\leq30$.
This allowed us to both test the veracity of Question~\ref{con:conjecture},
and to discover some surprising patterns. The six curves in
Figure~\ref{f:plot6} (unsurprisingly) correspond to the six possible choices for
$b=n\mod m$, but in a strange order \emph{viz.} $b=1,6,2,5,3,4$ going from
the highest curve to the lowest. (Incidentally, this observation
motivated our ``modulo~$m$'' proof of Theorem~1.) We noticed also that for many choices of
$m$ and~$b$ the sequence $f(am+b,m)$ for $0\leq a\leq 1000$ 
was strictly decreasing, or strictly increasing. However, for very few choices
e.g. $(m,b)=(26,24)$, the sequence initially increased (6 times), and
then increased (596 times) and then increased (397 times). (The graph
is a very flat sawtooth and so looks horizontal.)
These unusual patterns lead us to question the existence of
a simple proof of Question~\ref{con:conjecture}.
\end{remark}

Question~\ref{con:conjecture} is true in the very special
case when $p$ is a prime.

\begin{theorem}\label{T:monotone}
Let $p$ be a prime. The sequence $\left(f(ap+b,p)\right)_{a\geq0}$
increases strictly for $0\leq b\leq\lfloor\frac{p-1}{2}\rfloor$, and for $a\geq\frac{p-1}{2}$,
decreases strictly for $\lfloor\frac{p-1}{2}\rfloor<b\leq p-1$.
\end{theorem}

\begin{proof}
Write $n=ap+b$ where $0\leq b<p$. It follows from the closed
formula 
$\rho(n,p)=\prod_{i=1}^a(1-\frac{1}{ip})$ of~\cite[Lemma~I]{ref:ErdosTuran67}, that
$\rho(n+p,p)=\rho(n,p)(1-\frac{1}{(a+1)p})$. Hence
\[
  \frac{f(n+p,p)}{f(n,p)}=\left(1-\frac{1}{(a+1)p}\right)\frac{\left(\frac{n+p}{p}\right)^{1-\frac{\phi(p)}{p}}}{\left(\frac{n}{p}\right)^{1-\frac{\phi(p)}{p}}}
  =\left(1-\frac{1}{(a+1)p}\right)\left(1+\frac{1}{a+\frac{b}{p}}\right)^{\frac{1}{p}}.
\]

Fix $p$ and $b$. Our proof has two cases. Case 1 proves that
the above ratio is at least 1 for $0\leq b\leq\lfloor\frac{p-1}{2}\rfloor$, and 
Case 2 shows the ratio is at most~1 for $\lfloor\frac{p-1}{2}\rfloor<b<p$.

{\sc Case 1.} $0\leq b\leq\lfloor\frac{p-1}{2}\rfloor$. The above ratio is at
least 1 if and only if
\begin{equation}\label{E:goal1}
  \left(1-\frac{1}{(a+1)p}\right)^p\geq \frac{a+c}{a+c+1}
    \qquad\qquad\textup{where $c=\frac bp$.}
\end{equation}
Observe that $0\leq c<1$, and for $0\leq c_1,c_2<1$ we have
\begin{equation}\label{E:montone}
  \frac{a+c_1}{a+c_1+1}<\frac{a+c_2}{a+c_2+1}
\qquad\textup{if and only if $c_1<c_2$.}
\end{equation}
The left-hand side of~\eqref{E:goal1} is independent of $c$, and
by~\eqref{E:montone} the right-hand side of~\eqref{E:goal1} is largest when
$c$ equals $c_0:=\frac12(1-\frac1p)$. Set $x=(a+1)p$. Then
\[
  \frac{a+c_0}{a+c_0+1}=\frac{a+\frac12(1-\frac1p)}{a+\frac12(1-\frac1p)+1}
  =\frac{2ap+p-1}{2ap+3p-1}=\frac{2x-p-1}{2x+p-1}.
\]
Hence~\eqref{E:goal1} is true if for all $a\geq0$ and all integers $p\geq1$,
we have
\begin{equation}\label{E:inc}
  \left(\frac{x-1}{x}\right)^p\geq\frac{2x-p-1}{2x+p-1}
    \qquad\textup{for all real numbers $x\geq2$.}
\end{equation}

We now prove~\eqref{E:inc} by induction on $p$ for
all integers $p\geq1$. The case $p=1$
is clearly true. In the following display,
the first inequality follows from the inductive hypothesis, and the second
requires proof:
\[
\left(\frac{x-1}{x}\right)^{p+1}=\left(\frac{x-1}{x}\right)^p\left(\frac{x-1}{x}\right)\stackrel{\rm IH}{\geq}
\left(\frac{2x-p-1}{2x+p-1}\right)\left(\frac{x-1}{x}\right)
\stackrel{?}{\geq}\frac{2x-p-2}{2x+p}.
\]
The second inequality is equivalent to
\[
(2x+p)(2x-p-1)(x-1)\geq (2x-p-2)(2x+p-1)x.
\]
The left minus the right side is $p^2+p>0$.
This proves Case~1.

{\sc Case 2.} $\lfloor\frac{p-1}{2}\rfloor<b<p$.
In this case it suffices to prove
\begin{equation}\label{E:goal2}
  \left(1-\frac{1}{(a+1)p}\right)^p<\frac{a+c}{a+c+1}
    \qquad\qquad\textup{where $c=\frac bp$.}
\end{equation}
First note that $\lfloor\frac{p-1}{2}\rfloor+1=\lceil\frac{p}{2}\rceil$.
Hence $\frac{b}{2}\leq\lceil\frac{p}{2}\rceil\leq b$ and so $\frac12\leq c$.
For $c\geq\frac12$, the right-hand side of~\eqref{E:goal2} is
smallest for $c=\frac12$ by~\eqref{E:montone}.
As before, set $x=(a+1)p$. We prove~\eqref{E:goal2} by establishing the
inequality below:
\begin{equation}\label{E:G}
  \left(\frac{x-1}{x}\right)^p<\frac{a+\frac12}{a+\frac32}
  =\frac{2a+1}{2a+3}=\frac{2x-p}{2x+p}.
\end{equation}
Reasoning as in Case~1, we prove~\eqref{E:G} for $p\geq1$ by
induction on $p$. Certainly~\eqref{E:G} is true for $p=1$. Assume
it is true for some $p\geq1$. By the inductive hypothesis:
\[
  \left(\frac{x-1}{x}\right)^{p+1}=\left(\frac{x-1}{x}\right)^p\left(\frac{x-1}{x}\right)\stackrel{\rm IH}{<}\left(\frac{2x-p}{2x+p}\right)\left(\frac{x-1}{x}\right)
  \stackrel{?}{\leq}\frac{2x-p-1}{2x+p+1},
  \]
  where the last inequality is equivalent to
\[
  (2x-p)(x-1)(2x+p+1)\leq(2x-p-1)x(2x+p)\quad\textup{or}\quad p^2+p\leq 2x.
\]
Finally, $2x\geq p^2+p$ is true for $a\geq\frac{p-1}{2}$.
This proves Case~2, and the theorem.
\end{proof}


\end{document}